\documentclass[12pt,headings=big]{scrartcl}
 \usepackage{cmap}
 \usepackage{amsmath,amsxtra,amscd,amssymb,latexsym,stmaryrd,amsthm,lipsum}
 \usepackage{mathrsfs, calc, xcolor}

\usepackage[colorlinks=true,linkcolor=red!70!black,citecolor=green!70!black,
    urlcolor=magenta!70!black,backref]{hyperref}

  \usepackage[utf8]{inputenx}
  \usepackage[T1]{fontenc}

\addtokomafont{title}{\rmfamily\itshape}

\theoremstyle{plain}
\newtheorem{theo}{Theorem}[section]
\newtheorem*{theo*}{Theorem}
\newtheorem{coro}[theo]{Corollary}
\newtheorem{prop}[theo]{Proposition}
\newtheorem{lemm}[theo]{Lemma}
\newtheorem{theomain}[theo]{Theorem}

\theoremstyle{definition}
\newtheorem{defi}[theo]{Definition}
\newtheorem{rema}[theo]{Remark}
\newtheorem{exem}[theo]{Example}

\newcommand*{\dd}%
  {\relax\ifnum\lastnodetype>0\mskip\medmuskip\fi\mathrm{d}}
\newcommand{\fspace}[1]{\mathcal{#1}}
\newcommand{\fX}{\fspace{X}}

\newcommand{\op}[1]{\mathrm{#1}}
\newcommand{\one}{\boldsymbol{1}}

\newcommand{\mchain}{\mathsf}

\newcommand{\Lip}{\operatorname{Lip}}
\newcommand{\BV}{\operatorname{BV}}

\newcommand{\fXO}{\operatorname{\fspace{X}}(\Omega)}
\newcommand{\diam}{\operatorname{\mathrm{diam}}}

\newcommand{\spacem}[1]{\mathcal{#1}}

\newcommand{\im}{\operatorname{Im}}

\newcommand{\Hol}{\operatorname{Hol}}

\newlength{\hypobox}
\newlength{\gapbox}
\newcounter{hypo}
\renewcommand{\thehypo}{\textup{(H\arabic{hypo})}}
\newenvironment{hypo}{\refstepcounter{hypo}%
  \begin{list}{}{}\item[\textbf{\thehypo}]}{\end{list}}

\newcounter{hypop}
\renewcommand{\thehypop}{\textup{(H\arabic{hypop}')}}
\newenvironment{hypop}{\setcounter{hypop}%
                       {\value{hypo}-1}\refstepcounter{hypop}%
     \begin{list}{}{}\item[\textbf{\thehypop}]}{\end{list}}

\newcounter{gap}
\renewcommand{\thegap}{\textup{(SG\arabic{gap})}}
\newenvironment{gap}{\refstepcounter{gap}%
  \begin{list}{}{}\item[\textbf{\thegap}]}{\end{list}}

\title{Effective high-temperature estimates for intermittent maps}

\author{Beno\^{\i}t R. Kloeckner \thanks{Universit\'e Paris-Est, Laboratoire d'Analyse et de Mat\'ematiques Appliqu\'ees (UMR 8050), UPEM, UPEC, CNRS, F-94010, Cr\'eteil, France}}

\begin{document}

\maketitle

\begin{abstract}
Using quantitative perturbation theory for linear operators, we prove
spectral gap for transfer operators of various families of intermittent maps with almost constant potentials (``high-temperature'' regime). H\"older and bounded $p$-variation potentials are treated, in each case under a suitable assumption on the map, but the method should apply more generally. It is notably proved that for any Pommeau-Manneville map, any potential with Lispchitz constant less than $0.0014$ has a transfer operator acting on $\Lip([0,1])$ with a spectral gap; and that for any $2$-to-$1$ unimodal map, any potential with total variation less than $0.0069$ has a transfer operator acting on $\BV([0,1])$ with a spectral gap.
We also prove under quite general hypotheses that the classical definition of spectral gap coincides with the formally stronger one used in \cite{GKLM}, allowing all results there to be applied under the high temperature bounds proved here: analyticity of pressure and equilibrium states, central limit theorem, etc.
\end{abstract}

\section{Introduction}

The thermodynamical formalism, which provides a deep understanding of invariant measures of some topological dynamical systems, is by now quite well understood in the uniformly hyperbolic setting with sufficiently regular potentials. However, the non-uniformly hyperbolic setting presents many challenges and is under a lot of scrutiny. Here we shall restrict to ``intermittent'' maps, which are expanding in certain zones but have a neutral fixed point or even a contracting behavior in other regions.

For some intermittent maps such as the Pommeau-Manneville family, on the one hand the absolutely continuous measure only exhibits polynomial decay of correlation \cite{Hu, Sarig}, but on the other hand Ruelle-Perron-Frobenius\footnote{The term \emph{Gibbs measure} is often used, including in some of my previous works, but conflicts with many related but subtly different concepts.} measures of potentials with sufficiently small H\"older or $C^k$ norm (``high-temperature regime'') have exponential decay of correlation, following from a spectral gap for their transfer operator \cite{CV}.

In view of this diversity of behavior, it is desirable to better understand where lies the frontier between this two regimes, polynomial versus exponential decay of correlations. As a small step in this direction, we shall prove completely explicit lower bounds on the size of the high-temperature regime. These bounds are certainly far from being sharp for most single map we consider, but they are uniform over rather large families of maps.

The idea is simply to use spectral theory of perturbed linear operators: one only has to prove spectral gap for a \emph{constant} potential, which is usually easy, and then conclude by stability of the spectral gap property. The important additional point is that we use here an \emph{effective} perturbation theory, leading to effective bounds. The present approach also has the advantage of simplicity (see the proofs of Theorems \ref{theo:Hol} and \ref{theo:BV}).

\paragraph{Transfer operators.}

Recall that given a finite-to-one dynamical system $T:\Omega\to\Omega$, where $\Omega$ is a metric space, and a potential $\varphi:\Omega\to\mathbb{R}$ in a suitable Banach algebra of functions $\fX(\Omega)$, one defines a \emph{transfer operator} $\op{L}_{T,\varphi}$ which is a bounded linear operator from $\fX(\Omega)$ to itself, mapping $f$ to the sum (or average) of $e^\varphi f$ along inverse images under $T$. When $T$ is $k$-to-one for some $k\in\mathbb{N}$ we shall take the ``average'' point of view and set
\[\op{L}_{T,\varphi} f(x) = \frac{1}{k} \sum_{y\in T^{-1}(x)} e^{\varphi(y)} f(y).\]
More generally, we will consider the case when there is some probability transition kernel $\mchain{M}=(m_x)_{x\in\Omega}$ (i.e.  for all $x\in\Omega$ $m_x$ is a probability measure) such that $m_x$ is concentrated on $T^{-1}(x)$ (sometimes up to some ``negligible'' set, see Remark \ref{rema:RPF} below) and set
\[\op{L}_{\op{M},\varphi} f(x) = \int_\Omega  e^{\varphi(y)} f(y) \dd m_x(y).\]
The above $k$-to-one case is included in this framework by setting $m_x=\frac1k \sum_{y\in T^{-1}(x)} \delta_{y}$. The subscripts $\mchain{M}$ and $\mchain{T}$  will most often be kept implicit or interchanged depending on the focus we want to choose, and it should be clear from the context what the implicit object is.

The dual $\op{L}_\varphi^*$ of $\op{L}_{\varphi}$ preserves the set of finite positive measures, and under suitable assumptions $\op{L}_{\varphi}$ can be shown to have a maximal eigenvalue $\lambda_\varphi$, a unique positive eigenfunction $h_\varphi$ and $\op{L}_\varphi^*$ to have a unique positive eigenprobability $\nu_\varphi$. They can then be used to construct a $T$-invariant positive measure $\dd\mu_\varphi=h_\varphi \dd\nu_\varphi$ (the normalization of $h_\varphi$ being taken to ensure $\mu_\varphi$ is a probability) which we shall call the \emph{Ruelle-Perron-Frobenius} (RPF) measure of the potential $\varphi$.

The transfer operator is thus an important tool in the study of invariant measures, and its spectral property are crucial. In particular, if $\op{L}_{T,\varphi}$ has a spectral gap below $\lambda_\varphi$ one easily obtains an exponential decay of correlations:
\[ \Big\lvert \int f\circ T^n \cdot g \dd\mu_\varphi -\big(\int f\dd\mu_\varphi\big)\big(\int g\dd\mu_\varphi\big)\Big\rvert =O (e^{-a n}) \quad\forall f,g\in\fX(\Omega)\]
which means that if $X$ is a random point drawn according to $\mu_\varphi$, for any sufficiently regular $f,g$ the number $g(X)$  is ``almost independent'' from $f(T^n(X))$ when $n$ is large. This kind of properties have countless application in the study of ergodic properties of $T$, and it is thus desirable to know when a spectral gap occur.

\begin{rema}\label{rema:RPF}
The invariance of $\mu_\varphi$ is proved as follows. For all $f\in\fX(\Omega)$ (where $\fX(\Omega)$ is assumed to contain enough functions to approximate uniformly all continuous functions) it holds:
\begin{equation}
\int f \dd T_*(\mu_\varphi)
  = \int f\circ T\cdot h_\varphi \dd\nu_\varphi 
  = \int f\circ T\cdot h_\varphi \dd\Big(\frac{1}{\lambda_\varphi}\op{L}_\varphi^*\nu_\varphi\Big) 
  = \frac{1}{\lambda_\varphi}\int \op{L}_\varphi\big(f\circ T\cdot h_\varphi \big) \dd\nu_\varphi
\label{eq:invariance}
\end{equation}
where
\[
\op{L}_\varphi(f\circ T\cdot h_\varphi)(x) = \int e^{\varphi(y)} f(T(y)) h_\varphi(y) \dd m_x(y) = f(x) \int e^{\varphi(y)} h_\varphi(y) \dd m_x(y)
\]
the last equality using that $m_x$ is concentrated on $T^{-1}(x)$. But to plug this into \eqref{eq:invariance}, \emph{we only need this for $\nu_\varphi$-almost all $x$}. Then we get 
\[\int f \dd T_*(\mu_\varphi)
  = \frac{1}{\lambda_\varphi}\int f \, \op{L}_\varphi(h_\varphi) \dd\nu_\varphi = \int f \dd\mu_\varphi.
\]
The relevance of this relaxation comes from examples such as the Pommeau-Manneville family $(T_q)_{q>0}$ below, which are not strictly speaking $k$-to-one; see Example \ref{exem:subtle}.
\end{rema}

\paragraph{Intermittent maps.}

It is known in a broad generality that if $T$ is uniformly expanding (or hyperbolic) and $\varphi$ is H\"older, then there is a spectral gap for the transfer operator acting on H\"older functions. Investigations have thus turned toward non-uniformly expanding maps (or less regular potentials); one particular class of such maps often plays the role of test case: the Pommeau-Manneville family
\begin{align*}
T_q : [0,1] &\to  [0,1] \\
     x & \mapsto \begin{cases} x(1+(2x)^q) & \mbox{if } x\in[0,\frac12) \\
                               2x-1 &\mbox{if } x\in [\frac12,1] \end{cases}
\end{align*}
where $q\in (0,+\infty)$ is a parameter quantifying the tangency to identity at the neutral point $0$.\footnote{Variants acting on the circle, which can then be made $C^1$ with derivative of the form $1+O(x^q)$ at the fixed point $0$ might be preferred by some readers, and this makes no difference in our result.}
This type of maps is sometime called ``intermittent'', since the dynamics is close to the dynamics of a uniformly expanding map until the orbit gets close to the neutral fixed point: then the orbit takes a long time to go away from the fixed point.

Let us state a specialized version of our first main result (we obtain below similar bounds for all $\alpha$-H\"older potential, and for more general maps, possibly acting on higher-dimensional spaces).
\begin{theomain}\label{theo:mainHol}
For any $q>0$ and any potential $\varphi\in \Lip(\Omega)$ such that $\Lip(\varphi)\le 0.0014$, the transfer operator $\op{L}_{T_q,\varphi}$ has a spectral gap when acting on $\Lip(\Omega)$.
\end{theomain}
The independence to $q$ in the case of the Pommeau-Manneville family was already known (see the dependencies of $\varepsilon_\phi$ in \cite{CV}), and comes from the fact that they exhibit a backward \emph{average} contraction rate uniformly bounded by the presence of a strictly contracting branch. Note that for any value of the Lipschitz constant less than $0.0014$ the method provides an explicit lower bound for the size of the spectral gap.\\

Another important class of maps is the class of unimodal interval maps. We shall consider only maps with an (almost) constant number of inverse images, and thus restrict to unimodal maps which are $2$-to-$1$ (counted with multiplicity).

For this class, the bounded ($p$-)variation class of regularity will prove extremely efficient. There, the derivative plays little role, since bounded $p$-variation is ``insensitive to stretching''. We instead rely on its ``extensiveness'' (total variation is no more than the sum of variations over tiles of a tiling of the phase space) to deduce a general result from which we extract this simply stated version (we use $\BV([a,b])$ to denote the space of bounded variation function defined on $[a,b]$, and $\BV(\varphi)$ to denote the total variation of such a function $\varphi$).
\begin{theomain}\label{theo:mainBV}
Let $T:[a,b]\to [a,b]$ be a continuous interval map which is increasing from a subinterval $[a,c]$ onto $[a,b]$ and decreasing from $[c,b]$ onto $[a,b]$. For all potential $\varphi\in\BV([a,b])$ such that
\[\BV(\varphi) \le 0.0069\]
the transfer operator $\op{L}_{T,\varphi}$ has a spectral gap when acting on the space of bounded variation functions.
\end{theomain}
The main appeal of this result is that there is no contraction assumption. As before, we get a more general result, valid for $p$-BV potentials, which include in particular $1/p$-H\"older functions. It follows that we get a spectral gap for H\"older potential without any contraction assumption (and on a larger ball). The catch is that the spectral gap is in the $p$-BV norm; if $\varphi, f$ are $1/p$-H\"older with $\mu_\varphi(f)=0$ and $\lVert\varphi\rVert_{\BV_p}$ small enough, then
$\lVert \op{L}_\varphi^n(f)\rVert_{\BV_p}$ decays to zero exponentially fast (in particular the same holds for its uniform norm), but its H\"older norm could be arbitrarily large. However this is a rather small price to pay since many applications of spectral gap (e.g. exponential decay of correlations) do not rely so much on the norm, but more on the nature of the potentials one is allowed to consider.\\

To conclude this article, in Section \ref{sec:gaps} we prove that for transfer operators, having a spectral gap in the sense used in \cite{K:perturbation} (which is a common definition which we also use here) is equivalent to the seemingly stronger condition used in \cite{GKLM} (see Section \ref{sec:eigen}). In particular, under the assumptions of Theorem \ref{theo:mainHol} or \ref{theo:mainBV} we can apply the results of \cite{GKLM}: there is an analytic dependency of pressure and $\mu_\varphi$ on $\varphi$, the RPF measures in the high-temperature regime (in particular, the maximal entropy measure) satisfy a Central Limit Theorem, etc. 

As further applications of spectral gap, let us mention limit theorems:  in the setting of Theorems \ref{theo:mainHol} and \ref{theo:mainBV},
if $X_0$ is a random variable with law $\mu_\varphi$, $X_{k+1}:=T(X_k)$ and $\psi$ is an observable in the corresponding functional space ($\alpha$-H\"older or $p$-BV), then the random process with ``hidden deterministic dependency'' $(\psi(X_k))_{k\ge 0}$ satisfy a Law of Large Numbers, quantified by concentration inequalities, and the Central Limit Theorem, quantified by Berry-Esséen bounds. The proofs can be found for example in \cite{HH}, or in \cite{K:concentration} in a completely effective version (i.e. explicit non-asymptotic constants are obtained).

\section{Banach spaces and Wasserstein metric}

In this section we give the functional analytic set up in which we shall work. The phase space $\Omega$ will always be a metric space, and its metric will be denoted by $d(\cdot,\cdot)$.

\subsection{H\"older norms}

The H\"older classes are among the most usual regularity classes considered for potentials. For any $\alpha\in(0,1]$, one defines the $\alpha$-H\"older constant of a function $f:\Omega\to \mathbb{R}$ as the number
\[ \Hol_\alpha(f) := \sup_{x\neq y\in \Omega}\frac{\lvert f(x)-f(y)\rvert}{d(x,y)^\alpha}.\]
A function is $\alpha$-H\"older if its $\alpha$-H\"older constant is finite, and the set of such function form a Banach space when endowed with the usual H\"older norm. We prefer to assume $\Omega$ to be bounded and use the slightly modified norm
\[\lVert f\rVert_{\Hol_\alpha} := \lVert f\rVert_\infty + (\diam\Omega)^\alpha\Hol_\alpha(f)\]
which thanks to the $(\diam\Omega)^\alpha$ factor is more homogeneous: the two terms have the same physical dimension (they are to be expressed in the same physical unit as $f$), and the norm is invariant if one rescales the distance by a constant. This homogeneity will be slightly more natural and effective in the computations later on.

We denote by $\Hol_\alpha(\Omega)$ the space of $\alpha$-H\"older functions, which is in fact a Banach algebra, i.e. $\lVert fg\rVert_{\Hol_\alpha} \le \lVert f\rVert_{\Hol_\alpha} \lVert g\rVert_{\Hol_\alpha}$.

Indeed, for all $f,g\in \Hol_\alpha(\Omega)$ it holds $\lVert fg\rVert_\infty\le \lVert f\rVert_\infty \lVert g\rVert_\infty$ and
for all $x,y\in \Omega$:
\begin{align*}
\lvert f(x)g(x)-f(y)g(y)\rvert &\le \lvert f(x) g(x)- f(x) g(y)\rvert+\lvert f(x)g(y)-f(y)g(y)\rvert \\
  &\le \lVert f\rVert_\infty \Hol_\alpha(g) d(x,y)^\alpha+\Hol_\alpha(f) d(x,y)^\alpha\lVert g\rVert_\infty \\
\Hol_\alpha(fg) &\le \lVert f\rVert_\infty \Hol_\alpha(g)+\Hol_\alpha(f)\lVert g\rVert_\infty.
\end{align*}

%
%

\subsection{Wasserstein metric}\label{sec:wass}

The $1$-Wasserstein metric is defined on the set $\mathcal{P}(\Omega)$ of Borel probability measures on $\Omega$ by
\[W_1(\mu,\nu) = \inf_{\pi\in\Gamma(\mu,\nu)} \int_{\Omega\times \Omega} d(x,y) \,\dd\pi(x,y)\]
where $\Gamma(\mu,\nu)$ is the set of measures on $\Omega\times \Omega$ whose marginals
are $\mu$ and $\nu$. Elements of $\Gamma(\mu,\nu)$ are called \emph{transport plans} or \emph{couplings} from $\mu$ to $\nu$ For any $\alpha\in (0,1]$, $d(\cdot,\cdot)^\alpha$ is also a metric and the corresponding Wasserstein metric is denoted by $W_\alpha$.

We will only use a few basic properties: $W_\alpha$ is indeed a metric; the infimum in its definition is always attained by some transport plan, then called optimal and generally not unique;
the topology induced by $W_\alpha$ is the weak-$\ast$ topology as soon as $\Omega$ is compact. The Kantorovich duality enables a reformulation of $W_\alpha$ as
\[W_\alpha(\mu,\nu) = \sup_{f} \Big|\int f\,\dd\mu - \int f\,\dd\nu\Big|\]
where the supremum is on all functions $f:\Omega\to\mathbb{R}$ such that $\Hol_\alpha(f) = 1$.


General references on Transport Theory and the Wasserstein distance are \cite{Vi1} and \cite{Gigli:book}

\subsection{Bounded $p$-variation norms}
\label{sec:pBV}

When $\Omega$ is an interval or the circle, we can consider function of bounded variation or generalizations with great efficiency. For the sake of notational simplicity we shall only consider intervals here, but the case of circle is handled very similarly. Higher-dimensional analogues also exist, but are less elementary and have different properties (in particular they are no longer ``insensitive to stretching'').

Assume $\Omega\subset\mathbb{R}$ is an interval, fix $p\in[1,\infty)$ and let $f:\Omega\to\mathbb{R}$. An increasing sequence $x_0<x_1<\dots<x_n$ of elements of $\Omega$ is called a \emph{partition} and denoted by $\underline x$. The $p$-variation of $f$ along the partition $\underline{x}$ is defined by
\[ v_p(f,\underline{x}) := \Big(\sum_{j=1}^n \lvert f(x_j)-f(x_{j-1})\rvert^p \Big)^{\frac1p} \]
and the total $p$-variation of $f$ is defined by
\[ \BV_p (f) := \sup_{\underline{x}} v_p(f,\underline{x})\]
where the supremum is taken over all partitions of $\Omega$.
If this number is finite, $f$ is said to be of \emph{bounded $p$-variation}.

When $p=1$, this is the usual total variation, but $p>1$ is interesting because it allows for much more irregular potentials. Many $\alpha$-H\"older functions ($\alpha<1$) are not of bounded variation, but when $\Omega$ is bounded all $\alpha$-H\"older functions are of bounded $p$-variation for all $p\ge 1/\alpha$. Indeed
for $p\alpha=1$ we have 
\begin{align*}
v_p(f,\underline{x}) &= \Big(\sum_{j=1}^n \lvert f(x_j)-f(x_{j-1})\rvert^p\Big)^{\frac{1}{p}} \\
  &\le \Hol_\alpha(f) \Big(\sum_{j=1}^n \lvert x_j-x_{j-1}\rvert ^{\alpha p}\Big)^{\frac{1}{p}} \\
  &\le \Hol_\alpha(f) \Big(\sum_{j=1}^n x_j-x_{j-1}\Big)^{\frac{1}{p}} \\
  &\le \Hol_\alpha(f) (\diam\Omega)^\alpha
\end{align*}
for all partition $\underline{x}$. For larger $p$, simply use that when $\Omega$ is bounded, $\alpha$-H\"older functions are also $\beta$-H\"older for all $\beta<\alpha$. 

We endow the space $\BV_p(\Omega)$ of functions of bounded $p$-variation with the norm
\[\lVert f\rVert_{\BV_p} := \lVert f\rVert_\infty + \BV_p(f).\]
It is more usual to replace in the above definition the supremum norm by an integral norm, but:
\begin{itemize}
\item we don't want to give such a special role to Lebesgue measure in our context,
\item the two definitions yield equivalent norms,
\item the above choice norm makes $\BV_p(\Omega)$ a Banach algebra (the proof is pretty much the same as for the H\"older norm). 
\end{itemize}

Two important interlaced differences between the bounded variation and H\"older classes of regularity are that:
\begin{itemize}
\item  H\"older is \emph{intensive} while BV is \emph{extensive}, i.e. if one partitions $\Omega$ into two intervals $I_1,I_2$, then sum of the total variation of $f$ on $I_1$ and $I_2$ is not greater than the total variation of $f$ over $\Omega$ (it can be lesser if $f$ has a jump precisely at the interface between the two intervals),
\item the H\"older constant is diminished when one stretches the space, i.e. if a map $Y$ is $\theta$-contracting, then $\Hol_\alpha(f\circ Y)\le \theta^\alpha \Hol_\alpha(f)$, while $\BV$ is insensitive to stretching.
\end{itemize}
This is to be kept in mind to understand the difference in the assumptions of our main results in these two cases.

\subsection{Positive eigendata and spectral gap}
\label{sec:eigen}

Let us now precise the notion of spectral gap, which appears under slightly different wordings in the literature. We consider a bounded linear operator $\op{L}$ acting on a Banach space $\fX$. The operator norm is denoted, as the norm on $\fX$, by $\lVert\cdot\rVert$.

The first two definitions of a spectral gap will make sense for a general Banach space $\fX$, but the third one will need $\fX$ to be a space of function, which we therefore assume from now on. This immediately provides us with an additional structure: the cone of positive functions.
\begin{defi}\label{defi:positive}
We say that an operator $\op{L}$ \emph{has positive eigendata} if it has a positive eigenvalue $\lambda$ and a positive,
bounded away from $0$ eigenfunction $h\in\fX$ for this eigenvalue.
\end{defi}

Let us phrase three notions of spectral gap; the first two are classical and easily seen to be equivalent, and the third one is used in \cite{GKLM} and, under some general hypotheses, will be proved to be equivalent to the other two in Section \ref{sec:gaps}. We assume $\op{L}$ has an eigenvalue $\lambda>0$ with eigenvector $u_0$, and spectral gap is to be understood with respect to $\lambda$. 

\begin{gap}
There is a complement $G$ to $\langle u_0\rangle$ in $\fspace{X}$ (i.e. $G$ is a closed subspace and $\langle u_0\rangle\oplus G= \fspace{X}$) which is stable under the action of $\op{L}$, and there exist numbers $C>0$, $\delta\in(0,1)$ such that
\[\lVert \op{L}_{|G}^n\rVert \le C\lambda^n(1-\delta)^n.\]
\label{gap:weak}
\end{gap}

\begin{gap}
There exist two continuous linear operators $\op{P},\op{R}:\fspace{X}\to\fspace{X}$ such that $\op{L}=\lambda \op{P}+\op{R}$, $\op{P}$ is a rank one projection ($\dim\im \op{P}=1$ and $\op{P}^2=\op{P}$), $\op{P}\op{R}=\op{R}\op{P}=0$,  and there exists numbers $C>0$ and $\delta\in(0,1)$ such that:
\[\lVert \op{R}^n \rVert \le C\lambda^n(1-\delta)^n.\]
\label{gap:proj}
\end{gap}

The spectral gap assumption used in \cite{GKLM} seems slightly more precise:
\begin{gap}
The dual operator $\op{L}^*$ has an eigenmeasure
$\nu\in \spacem{P}(\Omega)$ for the eigenvalue $\lambda$, in particular
  \[\int \op{L}(f) \,\dd\nu = \lambda \int f \,\dd\nu \qquad
   \forall f\in\fX\]
and there exist positive constants
  $C>0$, $\delta\in(0,1)$ such that
  for all $n\in\mathbb{N}$ and all $f\in\fspace{X}$
   such that $\int f \,\dd\nu = 0$, we have
  \[ \lVert \op{L}^n(f) \rVert \le
    C \lambda^n(1-\delta)^n \lVert f\rVert.\]
\label{gap:GKLM}
\end{gap}

In particular, \ref{gap:GKLM} implies some identification between measures and linear forms on $\fX$, and identifies the stable complement $G$.

In all three definitions, $C$ is called the \emph{constant} and $\delta$ the \emph{size} of the spectral gap. In the main Theorems below, we shall use the first definition; this only matters when we are interested in the value of the constant, which may vary from a definition to another, but the mere statement that there is a spectral gap will hold for all definitions, 
thanks to Theorem \ref{theo:SG}.

\subsection{Quantitative perturbation Theory}
\label{sec:perturbation}

Let us now state the main perturbative tool we will use. We consider $\op{L}_0:\fX\to\fX$ a bounded linear operator of a Banach space $\fX$, with an eigenvalue $\lambda_0$, eigenvector $u_0$ and eigenform $\phi_0$ (i.e. $\phi_0$ is an eigenvector of the dual operator $\op{L}_0^*$ for the same eigenvalue). We assume a spectral gap, and consider the \emph{condition number} $\tau_0 := \frac{\lVert \phi_0\rVert \lVert u_0\rVert}{\lvert \phi_0(u_0)\rvert}$. We denote by $\pi_0$ the projection on $\ker\phi_0$ along the direction $\langle u_0\rangle$.

\begin{theo}[Corollary 2.12 of \cite{K:perturbation}]
\label{theo:perturbation}
If $\lambda_0=\lVert \op{L}_0\rVert=1$ and $\op{L}_0$ has spectral gap of size $\delta_0$ and constant $1$ (according to Definition \ref{gap:weak}), then all $\op{L}$ such that
\[\lVert\op{L}-\op{L}_0\rVert \le \frac{\delta_0(\delta_0-\delta)}{6(1+\delta_0-\delta)\tau_0\lVert\pi_0\rVert}\]
have a spectral gap of size $\delta$ below $\lambda_\op{L}$, with constant $1$. In particular, all $\op{L}$ such that
\[\lVert\op{L}-\op{L}_0\rVert < \frac{\delta_0^2}{6(1+\delta_0)\tau_0\lVert\pi_0\rVert}\]
have some spectral gap, with constant $1$.
\end{theo}

Results of this flavor are quite old, see e.g.\cite{Baumgartel}, \cite{DS} and   \cite{Kato}. However it is hard to find explicit such explicit radius bounds written for spectral gaps, and not only for having a simple isolated eigenvalue.

To apply this result, we will need to estimate the quantities $\delta_0$, $\tau_0$ and $\lVert\pi_0\rVert$.
We will first use the following lemma of Doeblin-Fortet/Lasota-Yorke type, which also appears in \cite{K:concentration} and that we reproduce with its proof for the sake of completeness.

Consider a normed space $\fX(\Omega)$ of (Borel measurable, bounded) functions $\Omega\to\mathbb{R}$, with norm $\lVert \cdot\rVert = \lVert \cdot \rVert_\infty+V(\cdot)$ where $V$ is a semi-norm.
\begin{lemm}\label{lemm:gap}
Assume that for some constant $D>0$, for all probability $\mu$ on $\Omega$ and for all $f\in\fX$ such that $\mu(f)=0$, 
$\lVert f\rVert_\infty \le D V(f)$.

Let $\op{L}_0\in\fspace{B}(\fX(\Omega))$ and assume that for some $\theta\in(0,1)$ and all $f\in\fX$: 
\[\rVert \op{L}_0f\rVert_\infty \le \lVert f\rVert_\infty \quad\mbox{and}\quad V(\op{L}_0 f)\le \theta V(f)\]
and having eigenvalue $1$ with an eigenprobability $\mu_0$, i.e. $\op{L}_0^*\mu_0=\mu_0$.

Then $\op{L}_0$ has a spectral gap (for the eigenvalue $1$, the contraction being on the stable space $\ker\mu_0$) with constant $1$, of size 
\[\delta_0 = \frac{1-\theta}{1+D\theta}\]
\end{lemm}

Of course, the hypotheses of the Lemma also ensure that $\lVert\op{L}_0\rVert \le 1$; and since we will only consider operators of the form
\[\op{L}_0f(x) = \frac1k \sum_{j=1}^k f(b_j(x))\]
we will always have $\op{L}_0\one=\one$ and thus $\lambda_0=\lVert\op{L}_0\rVert=1$.

\begin{proof}
Let $f\in \ker\mu_0$; then $\lVert \op{L}_0f\rVert_\infty\le \lVert f\rVert_\infty$ and $\op{L}_0f\in\ker\mu_0$, so that $\lVert \op{L}_0 f\rVert_\infty\le DV(\op{L}_0f)\le D\theta V(f)$.

Denote by $t\in[0,1]$ the number such that $\lVert f\rVert_\infty = t\lVert f\rVert$ (and therefore $V(f)=(1-t)\lVert f\rVert$).
The above two controls on $\lVert \op{L}_0(f)\rVert_\infty$ can then be written as $\lVert \op{L}_0(f)\rVert_\infty \le \min\big(t, D\theta(1-t)\big)\lVert f\rVert$ and using $V(\op{L}_0f)\le \theta V(f)$ again we get
\begin{align*}
\lVert \op{L}_0(f)\rVert &\le \min\big(t+\theta(1-t), (D+1)\theta(1-t)\big)\lVert f\rVert \\
\lVert (\op{L}_0)_{|\ker\mu_0}\rVert &\le \max_{t\in[0,1]} \min\big(t+\theta(1-t), (D+1)\theta(1-t)\big).
\end{align*}
The maximum is reached when $t+\theta(1-t) = (D+1)\theta(1-t)$, i.e. when $t=D\theta/(1+D\theta)$, at which point the value in the minimum is 
$(D+1)\theta/(D\theta+1) \in(0,1)$. Therefore there is a spectral gap with constant $1$ and size $1- (D+1)\theta/(D\theta+1)$, as claimed.
\end{proof}

\begin{lemm}\label{lemm:normofpi}
Assume again that for some constant $D>0$, for all probability $\mu$ on $\Omega$ and for all $f\in\fX$ such that $\mu(f)=0$, 
$\lVert f\rVert_\infty \le D V(f)$.

If the semi-norm $V$ is invariant under translation by a constant (i.e. $V(f+c)=V(f)$ for all $f\in\fX(\Omega)$ and all constant $c$), 
then all operators $\pi\in\fspace{B}(\fX(\Omega))$ of the form $\pi f=f-\mu(f)$ satisfy
\[\lVert\pi\rVert \le \frac{2D+2}{D+2}.\]
\end{lemm}

This lemma will apply below, as both $\Hol_\alpha$ and $\BV_p$ are invariant by translation by a constant.

\begin{proof}
The proof proceeds as the previous one. We consider $f\in\fX(\Omega)$ and let $t\in [0,1]$ be such that $\lVert f\rVert_\infty = t\lVert f\rVert$ and $V(f)=(1-t)\lVert f\rVert$. On the one hand we have $V(\pi f)=V(f)=(1-t)\lVert f\rVert$, and on the other hand we have both
\[\lVert \pi f\rVert_\infty \le 2\lVert f\rVert_\infty \qquad\mbox{and}\qquad \lVert \pi f\rVert_\infty \le DV(f),\]
the first inequality being a simple triangle inequality, while the second follows from an hypothesis since $\mu(\pi f)=0$.
Then we get $\lVert \pi f\rVert \le \min(1+t, (D+1)(1-t))\lVert f\rVert$ and therefore
\[\lVert \pi\rVert \le \max_{t\in[0,1]} \min (1+t, (D+1)(1-t)).\]
The maximum is reached at $t=D/(D+2)$, from which the result follows.
\end{proof}

\section{H\"older potentials and maps that are backward contracting on average}

Let $(\Omega,d)$ be a metric space which we assume to be of finite diameter.
\begin{defi}\label{defi:classH}
We say that a map $T:\Omega\to\Omega$ is of \emph{class $H(\alpha,\theta)$} where $\alpha\in(0,1]$ and $\theta\in(0,1)$ if there exist an integer $k\ge 2$ and ``inverse branches'' $(b_j)_{1\le j\le k}$ such that:
\begin{enumerate}
\item each $b_j$ is a (Borel) measurable map $\Omega\to\Omega$,
\item for all $x$ except possibly countably many, $T^{-1}(x) = \{b_1(x),b_2(x),\dots,b_k(x)\}$,
\item $T$ is \emph{backward $\theta$-contracting on average} with respect to the metric $d^\alpha$:
\[\forall y,z\in \Omega: \exists\sigma\in \mathfrak{S}_k,\quad \frac{1}{k} \sum_{j=1}^k d(b_j(y),b_{\sigma(j)}(z))^\alpha \le\theta d(y,z)^\alpha.\]
\end{enumerate}  
If $T$ is of class $H(\alpha,\theta)$, for each $\alpha$-H\"older potential we define a transfer operator by
\[\op{L}_{\varphi} f(x) = \frac1k \sum_{j=1}^k e^{\varphi(b_j(x))} f(b_j(x))\]
i.e. we consider the backward random walk $\mchain{M}=(\frac1k \sum_j \delta_{b_j(x)})_x$.
\end{defi}
The permutation $\sigma$ is meant to account for the fact that one may not be able to define inverse branches globally in a continuous fashion (hence the choice to define them globally but only as measurable maps). The second condition is not strictly necessary for our main Theorem to hold, but it ensures that as soon as the RPF measures are atomless, they are indeed $T$-invariant (see Remark \ref{rema:RPF}).

\begin{exem}\label{exem:subtle}
Let us consider three variations on the doubling map that only differ in the details that led us to the definition above:
\begin{align*}
D_c : \mathbb{R}/\mathbb{Z} &\to \mathbb{R}/\mathbb{Z} \\
x &\mapsto 2x \mod 1 \\
D_t : [0,1] &\to [0,1]
  & D_i : [0,1] &\to [0,1] \\
x &\mapsto \begin{cases} 2x &\mbox{if }x\in[0,\frac12] \\  
                            2-2x &\mbox{if } x\in [\frac12,1] 
                  \end{cases}
    & x &\mapsto \begin{cases} 2x &\mbox{if } x\in [0,\frac12) \\
                             2x-1 &\mbox{if } x\in [\frac12,1] 
               \end{cases}
\end{align*}
For each of them, let us define inverse branches showing that they these  three maps are of class $H(1,\frac12)$ (and thus $H(\alpha,\frac1{2^\alpha})$ for all $\alpha\in (0,1]$). For $x\in \mathbb{R}/\mathbb{Z}$, denote by $\{x\}$ its representative in $[0,1)$. Then we take:
\begin{itemize}
\item for $D_c$: $b_1 = x\mapsto \{x\}/2 \mod 1$ and $b_2 = x\mapsto \{x\}/2+1/2 \mod 1$. Both branches have one discontinuity point at $0$, while $D_c^{-1}(x)=\{b_1(x),b_2(x)\}$ for all $x$,
\item for $D_t$: $b_1 = x\mapsto x/2$ and $b_2 = x\mapsto 1-x/2$. Then both branches are continuous and $D_t^{-1}(x)=\{b_1(x),b_2(x)\}$ for all $x$, but $D_t$ is not $2$-to-one since $b_1(1)=b_2(1)$,
\item for $D_i$:  $b_1 = x\mapsto x/2$ and $b_2 = x\mapsto x/2+1/2$. Then both branches are continuous but $D_i^{-1}=\{b_1(x),b_2(x)\}$ only holds for $x\neq 1$, since $D_i(b_1(1))=D_i(1/2)=0$.
\end{itemize}
\end{exem}

Our first main result is the following.
\begin{theo}\label{theo:Hol}
Let $T$ be a map of class $H(\alpha,\theta)$ acting on a space $\Omega$ of finite diameter. For all $\alpha$-H\"older potentials $\varphi$ such that
\[\Hol_\alpha(\varphi) \le \frac{2}{3(\diam\Omega)^\alpha}\log \Big(1+ \frac{(1-\theta)^2}{16(1+\theta)}\Big)\]
the transfer operator $\op{L}_\varphi$ has a spectral gap (with constant $1$).
\end{theo}

\begin{proof}
We consider the transfer operator $\op{L}_0$ associated to the null potential, i.e.
\[\op{L}_0 f(x) = \frac{1}{k}\sum_{j=1}^k f(b_j (x)) \]
which under the above assumption acts continuously on $\Hol_\alpha(\Omega)$: indeed $\lVert \op{L}_0 f\rVert_\infty \le \lVert f\rVert_\infty$ and for all $x,y\in\Omega$ and some permutation $\sigma$ depending on $x$ and $y$ we have:
\begin{align*}
\big\lvert \op{L}_0f(x) - \op{L}_0f(y)\big\rvert &\le \frac{1}{k} \sum_{j=1}^k \big\lvert f(b_j(x))-f(b_{\sigma(j)}(y)) \big\rvert \\
  &\le \Hol_\alpha(f) \cdot \frac1k \sum_{j=1}^k d(b_j(x),b_{\sigma(j)}(y))^\alpha \\
  &\le \theta  \Hol_\alpha(f) d(x,y)^\alpha
\end{align*}
so that $\Hol_\alpha(\op{L}_0 f)\le \theta \Hol_\alpha(f)$. It follows that $\lVert \op{L}_0\rVert_{\Hol_\alpha}\le 1$.

The same computation shows that the dual operator $\op{L}_0^*$, which acts on probability measures over $\Omega$ since $\op{L}_0\one=\one$, is a $\theta$-contraction in the metric $W_\alpha$:
\begin{align*}
W_\alpha(\op{L}_0^*\mu,\op{L}_0^*\nu) &= \sup_{\Hol_\alpha(f)= 1} \Big\lvert \int \op{L}_0f \dd\mu -\int \op{L}_0f \dd\nu \Big\rvert \\
  &\le \sup_{\Hol_\alpha(f)= 1}\Hol(\op{L}_0f) W_\alpha(\mu,\nu) \\
  &\le \theta W_\alpha(\mu,\nu).
\end{align*}
Since the Wasserstein metric is complete, this ensures that there is exactly one probability measure which is invariant by $\op{L}_0^*$, which we denote by $\mu_0$.

Let us show that we can use Lemma \ref{lemm:gap}: if $\mu$ is any probability measure and $f\in\Hol_\alpha(\Omega)$ satisfies $\mu(f)=0$, 
then $f$ being continuous it must vanish at some point $x\in\Omega$.
Then for all $y\in\Omega$ we have $f(y)\le(\diam\Omega)^\alpha\Hol_\alpha(f)$.
It follows that we can apply Lemmas \ref{lemm:gap} and \ref{lemm:normofpi} with $D=1$, so that $\op{L}_0$ has a spectral gap of size $(1-\theta)/(1+\theta)$ with constant $1$ and $\lVert \pi_0\rVert \le\frac43$.

We shall apply Theorem \ref{theo:perturbation}. First observe that 
$\op{L}_\varphi = \op{L}_0(e^\varphi \cdot)$ so that, using the Banach Algebra property, we get
\[\lVert \op{L}_\varphi -\op{L}_0 \rVert_{\Hol_\alpha} = \lVert \op{L}_0((e^\varphi -\one)\cdot)\rVert_{\Hol_\alpha} \le \lVert \sum_{j=1}^\infty \frac{1}{j!}\varphi^j \rVert_{\Hol_\alpha} \le   \sum_{j=1}^\infty \frac{1}{j!}\lVert\varphi\rVert_{\Hol_\alpha}^j= e^{\lVert\varphi\rVert_{\Hol_\alpha}}-1. \]
Then, since $\one$ is the eigenfunction of $\op{L}_0$ and $\mu_0$ is a probability measure, we have $\tau_0=1$. Also observe that adding a constant to $\varphi$ simply multiply $\op{L}_\varphi$ by a scalar, not changing its spectral gap. 
Let  $\varphi\in\Hol_\alpha(\Omega)$ be a potential such that for some constant $c$ it holds
\[\lVert \varphi -c \rVert_{\Hol_\alpha} \le \log \Big(1+ \frac{\big(\frac{1-\theta}{1+\theta}\big)^2}{6\cdot\frac43(1+\frac{1-\theta}{1+\theta})}\Big)=\log \Big(1+ \frac{(1-\theta)^2}{16(1+\theta)}\Big);\]
Theorem \ref{theo:perturbation} then implies that
the transfer operator $\op{L}_\varphi=\op{L}_0(e^\varphi\cdot)$ has a spectral gap with constant $1$.

To conclude, we only have to observe that taking $c=(\sup \varphi+\inf \varphi)/2$, we must have $\lVert \varphi-c\rVert_\infty\le \frac12 (\diam\Omega)^\alpha \Hol_\alpha(\varphi)$ and thus $\lVert \varphi-c\rVert_{\Hol_\alpha} \le \frac32 (\diam\Omega)^\alpha \Hol_\alpha(\varphi)$.

\end{proof}

\begin{rema}
If we only want to know that there is some neighborhood of the line of constant potentials where the transfer operator has a spectral gap, without effective estimates, then we only need to prove a spectral gap for $\op{L}_0$ and dispense from the last half of the above proof.
\end{rema}

\begin{proof}[Proof of Theorem \ref{theo:mainHol}]
In the case of the Pommeau-Maneville family, we have $\diam\Omega=1$ and a contraction $\theta= \frac12+\frac{1}{2^{1+\alpha}}$ in $W_\alpha$ (one of the branch is $1/2$ contracting in the Euclidean metric, the other one is $1$-contracting). For $\alpha=1$, we obtain $\theta=3/4$ and observe
\[\frac{2}{3}\log\big(1+\frac{(1/4)^2}{16\times (7/4)}\big) \ge 0.0014.\]
\end{proof}

\section{Potentials with bounded $p$-variation}

For $\BV$ potentials and their relatives, we need much less geometric assumptions.
\begin{defi}\label{defi:classV}
We say that a map $T$ of a compact interval $\Omega\subset \mathbb{R}$ is of class $V$ if there are $k$ maps
$b_1,\dots,b_k:\Omega\to \Omega$ such that
\begin{enumerate}
\item each $b_j$ is monotonic from  $\Omega$ to an interval $I_j\subset\Omega$,
\item for all $x$ except possibly countably many, $T^{-1}(x) = \{b_1(x),b_2(x),\dots,b_k(x)\}$,
\item the $I_j$ have disjoint interior.
\end{enumerate}  
\end{defi}

As an example, any unimodular map whose restriction to both its monotonicity intervals is onto is of class $V$, irrespective of its derivative and of the behavior of its post-critical orbit.

Here again we use the definition of the transfer operators using the inverse branches:
\[\op{L}_{\varphi} f(x) = \frac{1}{k} \sum_{j=1}^k e^{\varphi(b_j(y)} f(b_j(x)),\]
keeping in mind that RPF measures are as expected $T$-invariant as soon as they are atomless.

\begin{theo}\label{theo:BV}
If $T$ is of class $V$, any potential $\varphi\in\BV_p(\Omega)$ such that
\[ \BV_p(\varphi) \le \frac23 \log\Big(1+ \frac{(k^{\frac1p}-1)^2}{16k^{\frac1p}(k^{\frac1p}+1)}\Big)\]
gives the transfer operator $\op{L}_{\varphi}$ acting on $\BV_p(\Omega)$ a spectral gap.
\end{theo}

Remark that this condition becomes \emph{less} stringent with large $k$; as $k\to\infty$ the bound goes to the limit $\frac23\log(1+1/16)\approx 0.04$. Meanwhile, for $p=1$ and $k=2$ we get the bound $\frac23\log(1+1/96)\ge 0.0069$. In particular Theorem \ref{theo:mainBV} follows immediately.

\begin{proof}
As above, we only need to consider the zero potential, whose corresponding transfer operator is again denoted by $\op{L}_0$. Given $f\in \BV_p(\Omega)$ it is again clear that $\lVert \op{L}_0 f\rVert_\infty \le \lVert f\rVert_\infty$; the core of the argument is then to prove that
\begin{equation}
\BV_p(\op{L}_0f) \le \frac{1}{k^{\frac1p}} \BV_p(f).
\label{eq:BV}
\end{equation}
This is easily achieved thanks to the ``extensiveness'' of the total $p$-variation: if $\underline{x} =(x_0<x_1<\dots<x_n)$ is any partition of $\Omega$, then
\[b_1(x_0) < \dots b_1(x_n) < b_2(x_0) < \dots b_2(x_n) <\dots < b_k(x_0) < \dots < b_k(x_n)\]
is a partition of $\Omega$, in the case when $b_j$ are all increasing and their images are in increasing order. In any other case, we only have to permute the $b_j$ and change the order of each subsequence $b_j(x_0), \dots, b_j(x_n)$ to obtain a partition, which we denote by
\[\underline{y} = (y_0 < \dots < y_{k(n+1)-1}).\]
Then using Minkowski's inequality and concavity of $x\mapsto x^{\frac1p}$ we have:
\begin{align*}
v_p(\op{L}_0 f,\underline{x})
  &= \frac{1}{k} \Big(\sum_{j=1}^n \Big\lvert \sum_{i=1}^k f(b_i(x_j)) - f(b_i(x_{j-1})) \Big\rvert^p \Big)^{\frac1p} \\ 
  &\le \frac{1}{k} \sum_{i=1}^k \Big(\sum_{j=1}^n \Big\lvert f(b_i(x_j)) - f(b_i(x_{j-1})) \Big\rvert^p \Big)^{\frac1p} \\ 
  &\le \Big(\frac1k \sum_{j=1}^n \sum_{i=1}^k \big\lvert f(b_i(x_j)) - f(b_i(x_{j-1}))\big\rvert^p \Big)^{\frac1p}\\
  &\le \frac{1}{k^{\frac1p}} v_p(f,\underline{y}) \le \frac{1}{k^{\frac1p}} \BV_p(f).
\end{align*}
By taking the upper bound over $\underline{x}$, we obtain \eqref{eq:BV}.

To apply Lemma \ref{lemm:gap}, observe that a $p$-BV function $f$ such that $\int f\dd\mu=0$ for some probability measure $\mu$ must be non-negative at some point and non-positive at some point, so that $\lVert f\rVert_\infty \le \BV_p(f)$.
The existence of an eigenprobability is ensured by the compactness of $\Omega$ and the fact that $\op{L}_0^*$ preserves the set of probability measures. We obtain that $\op{L}_0$ a spectral gap of size $(k^{\frac1p}-1)/(k^{\frac1p}+1)$ with constant $1$. Since $\BV_p(\Omega)$ is uniformly dense in the space of continuous functions, it follows that there is in fact a unique probability measure fixed by $\op{L}_0^*$.

The same computation as in the proof of theorem \ref{theo:Hol} shows that for all $\varphi\in\BV_p(\Omega)$, if there is a constant $c$ such that
\[\lVert \varphi -c \rVert_{\BV_p} \le \log \big(1+ \frac{(k^{\frac1p}-1)^2}{16k^{\frac1p}(k^{\frac1p}+1)}\big)\]
then the transfer operator $\op{L}_\varphi$ has a spectral gap. Taking $c=(\sup\varphi+\inf\varphi)/2$, we get as before $\lVert \varphi -c\rVert_{\BV_p} \le \frac32 \BV_p(\varphi)$ so that
a sufficient condition to have a spectral gap is
\[\BV_p(\varphi) \le \frac23 \log\Big(1+ \frac{(k^{\frac1p}-1)^2}{24k^{\frac1p}(k^{\frac1p}+1)}\Big).\]
\end{proof}

\section{Three shades of spectral gaps}\label{sec:gaps}

In this Section we show several possible definitions of \emph{spectral gap} are equivalent. This will enable us to ensure that the spectral gaps proved above actually match the needed assumption in \cite{GKLM}.

\subsection{Hypotheses on the function space}

We fix  a compact metric space $\Omega$ and space of functions $\fspace{X}(\Omega)$ endowed with a norm $\lVert\cdot\rVert$. As in \cite{GKLM}, we shall first assume the following hypothesis.
\begin{hypo}
$\fspace{X}(\Omega)$ is a Banach algebra of
Borel-measurable, bounded functions $\Omega\to\mathbb{R}$, which includes all constant functions, whose norm dominates the uniform norm (for some constant $A$ it holds $\lVert f \rVert \ge A\lVert f \rVert_\infty$) and such that for all positive, bounded away from zero function $f\in\fspace{X}(\Omega)$, the function $\log f$ also lies in $\fspace{X}(\Omega)$.
\label{hypo:Banach}
\end{hypo}
The last condition of $\log$ stability could be lifted up to changing the formulation of Theorem \ref{theo:SG} below (replacing $J\in\fX$ by $e^\varphi$ where $\varphi\in\fX$). The measurability of elements of $\fX$ and the domination condition $\lVert\cdot \rVert \ge A\lVert \cdot \rVert_\infty$ ensure every finite measure on $\Omega$ can be seen as a continuous linear form, i.e we have a natural map $C^0(\Omega)^*\to \fX(\Omega)^*$. We will also assume the following:
\begin{hypop}
If a finite measure $\mu$ on $\Omega$ is such that
  $\int f \dd\mu>0$ for all positive, bounded away from $0$ function $f\in\fspace{X}(\Omega)$, then $\mu$ must be a positive measure. Moreover
the set of finite positive measures is closed in $\fspace{X}(\Omega)^*$ (by which we mean that the natural map $C^0(\Omega)^*\to \fX(\Omega)^*$ has closed image).
\label{hypop:Banach}
\end{hypop}
This assumption means that $\fspace{X}(\Omega)$ is large enough to detect positivity of measures, and not to have too exotic linear forms which can be approximated by positive measures. This notably avoids trivial cases such as $\fspace{X}(\Omega)=\{\mbox{constants}\}$. To ensures one can identify finite measures with elements of $\fspace{X}(\Omega)^*$, we would need to ensure that the natural map $C^0(\Omega)^*\to \fX(\Omega)^*$ is injective. We will dispense from this, but by abuse still say that a form $\phi\in\fX(\Omega)^*$ ``is a finite (positive) measure'' whenever it is the image of a finite (positive) measure by the natural maps, irrespective of the uniqueness of this measure.

We will prove at the end of the Section that $\Hol_\alpha(\Omega)$ and $\BV_p([a,b])$ satisfy both hypotheses.

We only consider bounded linear operator $\fXO\to\fXO$ and, from now on, use $\fX$ as a shorthand for $\fXO$.

\begin{theo}\label{theo:SG}
Assume $\fX$ satisfies hypotheses \ref{hypo:Banach} and \ref{hypop:Banach} and let $\op{L}$ be an operator of the form
\[\op{L} f (x) = \int_\Omega J(y) f(y) \dd m_x(y)\]
for some positive function $J\in\fX$ and transition kernel $(m_x)_{x\in\Omega}$. Assume $\op{L}$ has positive eigendata $\lambda$, $h$. Then the three definitions \ref{gap:weak}, \ref{gap:proj} and \ref{gap:GKLM} are equivalent for $\op{L}$ (with the same gap size $\delta$, but possibly different constants $C$).
\end{theo}

Much of what follows is either classical or of classical flavor, but Theorem \ref{theo:SG} seems not obvious enough for us to dispense from a proof, and we need some slightly non-standard version of common concepts (e.g. we will use cones that are not needed to be convex, but are open). Note that we work with real Banach spaces; the above definitions of a spectral gap, e.g. \ref{gap:weak}, are nonetheless well-known to be equivalent to the \emph{complex} spectrum to be contained in $\mathbb{D}_{\lambda(1-\delta)}\cup\{\lambda\}$, where $\mathbb{D}_r$ is the closed disk in the complex plane, centered at $0$ and of radius $r$.

\begin{proof}
Since \ref{gap:weak} is well-known to be equivalent to \ref{gap:proj} and \ref{gap:GKLM} clearly implies \ref{gap:weak} under our hypotheses (simply take $G=\ker\mu$), the part of Theorem \ref{theo:SG} that needs a proof is that \ref{gap:weak} implies \ref{gap:GKLM}. Assume $\op{L}$ satisfies \ref{gap:weak}.

The dual operator $\op{L}^*$ acting on $\fspace{X}^*$ by
$\op{L}^*\phi(u) := \phi(\op{L}u)$ preserves the line $G^\perp:=\{\phi\in\fspace{X}^* \mid \ker\phi\supset G\}$ and thus has an eigenvector $\phi_0$, which up to a multiplicative constant encodes the projection on $\langle u\rangle$ along $G$; moreover $\phi_0$ has the same eigenvalue $\lambda$ as $u_0$.

Similarly, $\op{L}^*$ must preserve the hyperplane $u_0^\perp:=\{\phi\in\fspace{X}^* \mid \phi(u_0)=0\}$ and one checks easily \ref{gap:weak} for $\op{L}^*$, with $u_0^\perp$ in the role of $G$ and the same spectral gap $\delta$. In particular, $\lambda$ is a simple isolated eigenvalue of $\op{L}^*$. Moreover 
for all $\phi\in\fspace{X}^*$ and all $u\in\fspace{X}$, writing $u=au_0+g$ its decomposition along $\langle u\rangle\oplus G$, we have
\begin{align*}
\Big(\frac{1}{\lambda^n}\op{L}^{*n}\Big)\phi(au_0+g) 
  &=\phi\Big(au_0+\frac{1}{\lambda^n}\op{L}^{n}g \Big) \\
  &\to a\phi(u_0)
\end{align*}

Now, to apply this in our context we will want $\phi_0$ to be a measure, even if $\fspace{X}^*$ is larger than the space of measures (which often happens since $\fspace{X}$ usually is quite smaller than the space of continuous functions). We will also need that $\phi_0$ be a \emph{positive} measure, let us thus use the additional structure of $\fspace{X}$ as a space of real functions. We say that $P\subset\fspace{X}$ is a \emph{cone} if it is stable under multiplication by all positive scalar. We say it is a \emph{positive cone} if it is a cone, is open in the subspace it generates, and $\overline{P}\cap(-P)=\varnothing$. The openness assumption is to be noted: it is really used below, and somewhat restricts the scope of application of this method (e.g. in the space of $L^1$ real functions, the subset of positive functions is not open but generates the whole space).

Any positive cone has a dual
$P^*:=\{\phi\in\fspace{X}^* \mid \forall u\in P, \phi(u)> 0\}$
which is a cone of $\fspace{X}^*$ (but we do not claim it to be a positive cone) whose closure is, as soon as $P^*$ is not empty, $\overline{P^*}=\{\phi\in\fspace{X}^* \mid \forall u\in P, \phi(u)\ge 0\}$.

\begin{lemm}\label{lemm:duality}
If $P$ is a positive cone of $\fspace{X}$ which is preserved by $\op{L}$,
then $P^*$ is preserved by $\op{L}^*$. If additionally $u_0\in P$, then $G\cap P=\varnothing$ and up to multiplying it by $-1$, $\phi_0\in P^*$.

If additionally $\fspace{Y}\subset \fspace{X}^*$ is a subspace preserved by $\op{L}^*$ such that $\overline{P^*}\cap\fspace{Y}$ is closed in $\fspace{X}^*$ then either $\fspace{Y}\subset u_0^\perp$ or $\phi_0\in P^*\cap \fspace{Y}$.
\end{lemm}

\begin{proof}
Let $\phi\in P^*$ and $u\in P$. Then $\op{L}^*\phi(u)=\phi(Lu)$; since $Lu\in P$ this is a positive number. Thus $\op{L}^*\phi\in P^*$, and $\op{L}^*$ preserves $P^*$.

Assume now $u_0\in P$. If there exist some $g\in G\cap P$, then since $P$ is open in the space it generates there is some $a<0$ such that $au_0+g\in P$. Then for all $n$ we have $\lambda^{-n}\op{L}^n(au_0+g)\in P$, which converges to $au_0\in -P$. But that would give an element of $\overline{P}\cap -P$, which is excluded by the definition. Thus $G\cap P=\varnothing$.

Next we need to prove that choosing well the dual eigenvector,
for all $u\in P$ it holds $\phi_0(u)>0$.
By construction, up to a multiplicative constant we can assume that
$\phi_0(au_0+g)=a$ for all $a\in\mathbb{R}$ and $g\in G$.
If $u=au_0+g$ is in $P$, then $a<0$ is excluded as otherwise
$au_0=\lim \lambda^{-n}\op{L}^nu$ would be an element of $\overline{P}\cap -P$; and $a=0$ is excluded by $G\cap P=\varnothing$. Therefore $a>0$ and
$\phi_0\in P^*$.

Last, assume there exists $\phi\in\fspace{Y}\setminus u_0^\perp$ and write $\phi=b\phi_0+\psi$ where $\psi\in u_0^\perp$. Then
\[\Big(\frac{1}{\lambda^n}\op{L}^{*n}\Big)\phi = b\phi_0+\Big(\frac{1}{\lambda^n}\op{L}^{*n}\Big)\psi \to b\phi_0.\]
Since $\overline{P^*}\cap\fspace{Y}$ is closed and $b$ is nonzero (as $\phi\notin u_0^\perp$), we conclude that $\phi_0\in\fspace{Y}$. We already know it is in $P^*$, finishing the proof.
\end{proof}

Let us now finish the proof of Theorem \ref{theo:SG}: we consider the set $P$ consisting in all positive functions $\Omega\to\mathbb{R}$ in $ \fX$ which are bounded away from $0$. It is a cone, and since
the norm dominates the uniform norm it is open in $\fX$. An element of $\overline{P}\cap(-P)$ would be a non-negative and positive function, thus $P$ is a positive cone. The dual cone $P^*$ is the set of continuous linear forms of $\fX$ which are positive on all elements of $P$; denoting by $\fspace{Y}$ the subspace of $\fX^*$ consisting of finite measures, assumption \ref{hypop:Banach} ensures that $\overline{P^*}\cap \fspace{Y}$ is closed in $\fspace{X}^*$. The form of $\op{L}$ ensures that $\fspace{Y}$ is preserved by $\op{L}^*$. It is not true that $\fspace{Y}\subset h^\perp$, therefore
Lemma \ref{lemm:duality} ensures that $\op{L}^*$ has an eigenvector in $P^*\cap\fspace{Y}$, which by \ref{hypop:Banach} must be a positive measure. It is not zero since $\phi_0(u_0)\neq 0$, and up to proper scaling we can make it a probability measure. Then it is easily checked that $G$ must be the kernel of this eigenprobability, and \ref{gap:GKLM} holds.
\end{proof}

\begin{prop}
Let $\Omega$ be a compact metric space and $\fX$ be a Banach space of functions $\Omega\to\mathbb{R}$ satisfying \ref{hypo:Banach}. If $\fX$ is dense (for the uniform norm) in the space $C^0(\Omega)$ of continuous functions, then $\fX$ also satisfies \ref{hypop:Banach}.
\end{prop}

\begin{proof}
We prove a slight strengthening of the first property.
Let $\mu$ be a finite measure on $\Omega$ such that for all $f\in\fX$ bounded away from $0$, $\mu(f) := \int f\dd \mu \ge 0$ and let $g$ be any continuous positive function. Let $f_n\in \fX$ be such that $\lVert g-f_n\rVert_\infty\to 0$; since $\mu$ is finite, it follows that $\mu(f_n)\to\mu(g)$. By compactness $\inf g>0$ thus for $n$ large enough, $f_n$ is bounded away from $0$ and $\mu(f_n)\ge 0$. It follows that $\mu(g)\ge 0$, and $\mu$ is a positive measure.

Now let $(\mu_n)_{n\in\mathbb{N}}$ be a sequence of finite positive measures, and assume there is a continuous linear form $\phi\in\fX^*$ such that $\lVert\mu_n-\phi\rVert^*\to0$ (where $\mu_n$ are identified with elements of $\fX^*$ and $\lVert\cdot\rVert^*$ is the dual norm on $\fX^*$). We have to prove that $\phi$ is a finite positive measure.

First, $\lim \mu_n(\one) = \phi(\one)\in\mathbb{R}$ ensures that the total mass of $\mu_n$ is bounded by some $M>0$, uniformly in $n$, i.e. for all bounded measurable $f$ (in particular, $f\in\fX$) it holds $\lvert\mu_n(f)\rvert \le M\lVert f\rVert_\infty$ (positivity of the $\mu_n$ is crucial here).

Consider any non-zero $f\in\fX$. There exist an $n\in\mathbb{N}$ such that $\lVert\mu_n-\phi\rVert^*\le \lVert f\rVert_\infty/\lVert f\rVert$, and then
\begin{align*}
\lvert \phi(f)\rvert &\le \lvert \phi(f)-\mu_n(f)\rvert + \lvert \mu_n(f)\rvert \\
  &\le (1+M)\lVert f\rVert_\infty
\end{align*}
so that $\phi$ is continuous in $\lVert\cdot\rVert_\infty^*$ norm on $\fX$. It then admits an extension $\mu \in C^0(\Omega)^*$, which by the Riesz representation theorem is a finite measure. Since $\mu_n(f)\to \mu(f)$ for all $f\in\fX$ and the $\mu_n$ are positive, we have $\mu(f)\ge 0$ for all non-negative $f\in\fX$. In particular we can apply the strengthened first property proved above to deduce that $\mu$ is positive. In other words, $\phi$ identifies with a finite positive measure, as needed.
\end{proof}

It follows easily (using that $\log$ is locally Lipschitz and that $\Hol_\alpha(\Omega)$ is uniformly dense in the space of continuous fonctions \cite{Georganopoulos}):
\begin{coro}
For all compact metric space $\Omega$, the Banach space $\Hol_\alpha(\Omega)$ satisfies \ref{hypo:Banach} and \ref{hypop:Banach}.

For all bounded interval $I\subset\mathbb{R}$, the space $\BV_p(I)$ satisfies \ref{hypo:Banach} and \ref{hypop:Banach}.
\end{coro}

\bibliographystyle{amsalpha}
\bibliography{perturbation}

\end{document}